\documentclass{amsart}
\usepackage[T1]{fontenc}
\usepackage{times}
\usepackage{amssymb,verbatim}
\theoremstyle{plain}
\usepackage[T1]{fontenc}
\usepackage[utf8]{inputenc}
\usepackage[english]{babel}
\usepackage{amssymb}
\usepackage{amsthm}
\usepackage{graphics}
\usepackage{amsmath}
\usepackage{amstext}
\usepackage{multirow}
\usepackage{graphicx}
\usepackage{amsmath}
\usepackage{amssymb}
\usepackage{amsthm}
\usepackage{amstext}
\usepackage{mathrsfs}
\usepackage{amscd}
\usepackage{mathtools} 
\usepackage{hyperref}
\usepackage{tikz}
\usetikzlibrary{matrix}
\usepackage[all]{xy}
\usepackage{algorithm}
\usepackage{algorithmic}
\usepackage{tikz-cd}

\usepackage{verbatim}
\usepackage{fancyvrb}
\usepackage{xcolor}
\usepackage{listings}
\definecolor{codegreen}{rgb}{0, 0.6, 0}
\definecolor{codegray}{rgb}{0.5, 0.5, 0.5}
\definecolor{codepurple}{rgb}{0.58, 0, 0.82}
\definecolor{backcolour}{rgb}{0.95, 0.95, 0.92}

\lstdefinelanguage{Macaulay2}
{
xleftmargin=.2in, 
xrightmargin=.2in, 
basicstyle={\ttfamily}, 
keywordstyle={\color{blue}}, 
commentstyle={\color{codegreen}}, 
stringstyle={\color{red!40!black}}, 
rulecolor=\color{yellow}, 
basewidth={1.2ex}, %workaround for prompts being the same width as normal tt text
sensitive=false, 
morecomment=[l]{--}, 
morecomment=[s]{-*}{*-}, 
morestring=[b]", 
escapechar={`}, 
escapebegin={\rmfamily}, 
morekeywords={load, random, degree, genus, topComponents, ideal, Ext, minors, quotient, intersect, map, kernel, preimage, codim, sheaf, matrix, hilbertPolynomial, Projective, false, sheafExt, ann, cooker, flatten, gens, entries, basis, apply, return, select, subsets, sort, unique, then, else, toList, number, max, sum, det,join,min,and }
}
\lstalias{Macaulay2output}{Macaulay2}

\newtheorem{theorem}{Theorem}[section]
\newtheorem{MainTheorem}{Theorem}

\newtheorem{corollary}[theorem]{Corollary}

\newtheorem{lemma}[theorem]{Lemma}
\newtheorem{proposition}[theorem]{Proposition}
\theoremstyle{definition}
\newtheorem{definition}{Definition}[section]

\newtheorem{remark}[definition]{Remark}

\usepackage[all]{xy}

\newcommand{\Log}{\textsf{Log}}
\newcommand{\LPB}{\textsf{LPB}}

\newcommand{\Folx}[3]{{\sf Fol}^{#2}(#1, #3 )} %%Fol^codim_X(linebundle)

\newcommand{\mb}{\mathbb}
\newcommand{\mc}{\mathcal}

\newcommand{\C}{\mb C}
\newcommand{\Z}{\mb Z}

\newcommand{\extp}{\textstyle\bigwedge}
\newcommand{\lra}{\longrightarrow}
\newcommand{\tang}{\mathrm{tang}}

\DeclareMathOperator{\im}{Im}

\DeclareMathOperator{\codim}{codim}
\DeclareMathOperator{\sing}{sing}

\DeclareMathOperator{\Pic}{Pic}

\DeclareMathOperator{\Mor}{Mor}

\DeclareMathOperator{\rest}{rest}
\DeclareMathOperator{\rk}{rk}

\newcommand{\Pj}{\mb P}
\newcommand{\pn}{{\mb{P}^{n}}}
\newcommand{\p}[1]{{\mb{P}^{#1}}}

\newcommand{\op}[1]{{\mc O}_{\mb{P}^{#1}}}
\newcommand{\opn}{{\mc O}_{\mb{P}^{n}}}
\newcommand{\om}[2]{\Omega_{\mb{P}^{#2}}^{#1}}
\newcommand{\omn}[1]{\Omega_{\mb{P}^{n}}^{#1}}

\newcommand{\qn}{{Q^{n}}}
\newcommand{\q}[1]{{Q^{#1}}}

\newcommand{\oq}[1]{{\mc O}_{Q^{#1}}}

\newcommand{\TX}{{TX}}

\newcommand{\OX}{{\mathcal O}_X}
\newcommand{\omx}[1]{\Omega_{X}^{#1}}
\newcommand\restr[2]{{
  \left.\kern-\nulldelimiterspace % automatically resize the bar with \right
  #1 % the function
  \vphantom{\big|} % pretend it's a little taller at normal size
  \right|_{#2} % this is the delimiter
  }}

\def\OO{{\mathcal{O}}}
\def\CC{{\mathbb{C}}}

\def\ZZ{{\mathbb{Z}}}

\newcommand{\F}{\mc F}
\newcommand{\G}{\mc G}
\newcommand{\Ho}{\mc H}
\newcommand{\D}{\mc D}

\newcommand{\TF}{{T_\F}}
\newcommand{\TG}{{T_\G}}

\newcommand{\nF}{{N_\F}}
\newcommand{\nG}{{N_\G}}
\newcommand{\nH}{{N_{\Ho}}}

\newcommand{\aut}{ \mathfrak{aut}}

\newcommand{\fix}{ \mathfrak{fix}}
\newcommand{\rad}{{\rm rad}} %% radial vector field
\newcommand{\GL}{\rm GL}

\newcommand{\del}[1]{\frac{\partial}{\partial #1}}

\numberwithin{equation}{section}

\usepackage{mathtools}
\usepackage{ifmtarg}

\addtocounter{section}{0}             % Start with section 1
\numberwithin{equation}{section}       % Number formulas within sections
\title[Degree-one foliations on complete intersections]{Degree-one foliations on complete intersections}
\author[M. Figueira]{Mateus Figueira}
\address{Departamento de An\'alise Matem\'atica \\ Instituto de Matemática e Estatística \\ Universidade do Estado do Rio de Janeiro \\ Rua São Francisco Xavier, 524, Maracanã, 20550-900, Rio de Janeiro, Brazil. }
\email{mateus.figueira@uerj.br}

\author[C. Kuster]{Crislaine Kuster}
\address{ IMPA, Estrada Dona Castorina 110, Rio de
  Janeiro, 22460-320, Brazil }
\address{Institut de Mathématiques de Bourgogne, UMR 5584 CNRS,
Université de Bourgogne, 9 Avenue Alain Savary,
F-21000, Dijon, France}
\email{crislainekeizy@gmail.com}

\author[R. Lizarbe]{Ruben Lizarbe}
\address{Departamento de Estrutura Matem\'atica \\ Instituto de Matemática e Estatística \\ Universidade do Estado do Rio de Janeiro \\ Rua São Francisco Xavier, 524, Maracanã, 20550-900, Rio de Janeiro, Brazil. }
\email{ruben.monje@ime.uerj.br}

\author[A. Muniz]{Alan Muniz}
\address{Departamento de Matem\'atica \\ Centro de Ci\^encias Exatas e da Natureza \\ Universidade Federal de Pernambuco \\ Recife - PE, CEP 50740-560, Brasil}
\email{alan.nmuniz@ufpe.br}

\date{July 2026}

\makeatletter
\@namedef{subjclassname@2020}{\textup{2020} Mathematics Subject Classification}
\makeatother
\subjclass[2020]{Primary: 37F75, 32S65; Secondary: 14M25}

\keywords{holomorphic foliations, holomorphic distributions, complete intersections, hypersurfaces, extensions, restrictions}

\thanks{
{We warmly thank the anonymous referee for many comments and corrections that helped us improve our work.} Figueira was partially supported by a postdoctoral fellowship from the CAPES-PIBID program in Mathematics at the Universidade Federal Fluminense (UFF). Kuster was supported by CNPq Grant Number 140986/2021-9 and  CAPES - PDSE - Programa de Doutorado Sanduíche no Exterior Grant Number 88881.846472/2023-01.  Lizarbe was supported by CNPq/MCTI/FNDCT No.\ 18/2021 (``Global geometry of connections, holomorphic webs and foliations,'' Grant Number 402936/2021-3) and also by CNPq/MCTI No.\ 10/2023 – Edital Universal (``Geometria das equações diferenciais algébricas,'' Grant Number 408687/2023-1). {Muniz is partially supported by CNPq, grant number 305599/2026-7.}
}

\begin{document}

\begin{abstract}
We prove that, under mild restrictions, the space of codimension-one foliations of degree one on a smooth projective complete intersection has two irreducible components of logarithmic type. We also prove that the same conclusion holds for any smooth hypersurface of dimension at least three that is not a quadric threefold. The proofs of these results follow essentially from a more general structure theorem for foliations on manifolds covered by lines. 
\end{abstract}

\maketitle

\setcounter{tocdepth}{1}

\tableofcontents

\section{Introduction}

A \emph{singular holomorphic foliation $\F$ of codimension one} on a {smooth} projective variety $X$ is determined by a $1$-form $\omega$ with values in a line bundle $N$, satisfying the Frobenius integrability condition $\omega \wedge d\omega=0$, and whose singular set has codimension at least two. In this case, $N$ is called the \emph{normal bundle} of $\F$. The set of all such foliations with fixed normal bundle $N$ defines a quasi-projective variety, denoted by $\Folx{X}{1}{N}$, known as the \emph{space of codimension-one foliations on $X$ with normal bundle $N$}{, see Section \ref{Sec-Fol}}.

The irreducible components of these spaces have been studied in various contexts. For example, when $X$ is a surface, the integrability condition is automatically satisfied, and the corresponding spaces are irreducible. When $\Pic(X) \simeq \Z$, a foliation $\F$ with normal bundle $N$ is said to have \emph{degree} $d$ if $N \simeq \OX(d+2)$, where $\OX(1)$ denotes the hyperplane bundle. {The notion of degree is also defined for foliations of arbitrary codimension; see Section \ref{Sec-Fol}. } 

% In this case, we denote the space of degree-$d$ foliations by $\Folx{X}{1}{d}$.  For $X = \p{n}$,  $n \geq 3$, the cases $d = 0, 1, 2$ are completely classified, and partial results exist for $d = 3$ (see \cite{Deserti-Cerveau}, \cite{JP-Jouanoulou79}, \cite{LNC-96}, and \cite{CLP-2022}). {Explain Joanolou's Theorem}
% For more general smooth projective varieties, few results describe the structure of foliations of low degree. When $d = 0$, classifications are available for certain homogeneous varieties \cite{BFM, kuster2025}, for hypersurfaces \cite{Fig23}, and for smooth complete intersections \cite{ACM-dist}.

{
In this case, we denote the space of degree-$d$ codimension-one foliations by $\Folx{X}{1}{d}$.  For $X = \p{n}$,  $n \geq 3$, the cases $d = 0, 1, 2$ are completely classified, and partial results exist for $d = 3$. It is not clear who established the degree-zero case, but a proof can be found in \cite{Deserti-Cerveau}. In \cite{JP-Jouanoulou79}, Jouanolou proved that a degree-one foliation $\F$, then called a Jacobi equation, falls into two cases: 
\begin{enumerate}
    \item $\F$ is a linear pullback of a degree-one foliation on $\p2$; or 
    \item $\F$ is given by a pencil of quadrics with a double plane.
\end{enumerate}
Then $\Folx{\pn}{1}{1}$ has two irreducible components. Moreover, the general elements in these components are defined by \emph{logarithmic forms} (see \S\ref{S:components}). Theorem \ref{T: Main components} below extends this result to complete intersections.

In \cite{LNC-96}, Cerveau and Lins Neto described $\Folx{\pn}{1}{2}$: it has six irreducible components; four of logarithmic type, one composed of linear pullback foliations, and the so-called exceptional component. In \cite{CLP-2022}, the irreducible components of $\Folx{\pn}{1}{3}$ were partially classified; there are at least 24 irreducible components. Beyond the case of projective spaces, irreducible components of spaces of foliations have also been studied on other classes of varieties, including Fano threefolds \cite{LPT13}, homogeneous spaces \cite{BFM,kuster2025}, hypersurfaces and complete intersections \cite{ACM-dist, Fig23}, varieties over fields of positive characteristic \cite{mendson2024codimension,MP25-positive-char}, and linear pullback components of spaces of higher-codimensional foliations on toric varieties \cite{AMV23-fol-on-toric,velazquez2022toric}.

In the case of smooth projective hypersurfaces $X$ in $\p{n+1}$, $n\ge 3$, if 
\[
\deg(X)>2\deg(\F)+1,
\]
then the restriction of foliations from the ambient $\p{n+1}$ defines an isomorphism between $\Folx{X}{1}{d}$ and $\Folx{\p{n+1}}{1}{d}$, see \cite[Theorem 3.6]{Fig23}. Therefore, the previous classifications for foliations on projective spaces are valid for these varieties. The problem is then to study low-degree hypersurfaces; for degree-one foliations, we need to analyze hypersurfaces of degree at most three.

In \cite{LPT13}, Loray, Pereira, and Touzet proved that for $X\subset \p4$ a smooth quadric hypersurface, $\Folx{X}{1}{1}$ has three irreducible components; only two are restrictions from foliations on $\p4$. In this work, we prove that, for hypersurfaces, this is the only case where $\Folx{X}{1}{1}$ has more than two irreducible components. More generally, we study the irreducible components of the space of degree-one {codimension-one} foliations on smooth projective complete intersections. Our main result is the following:

}

% {
% \begin{MainTheorem} \label{T: Main components}
% Let $X$ be a smooth complete intersection in $\p{n+s}$ of type $(d_1,\ldots,d_s)$, such that either
% \begin{enumerate}
%     \item $X$ is a hypersurface, i.e. $s=1$, $n\geq 4$, and $c_1(TX)\geq 2$; or
%     \item $d_j \geq 3$ for all $j=1,\dots,s$, $n\geq 4$, and $c_1(TX)\geq 2$; or
%     \item $d_j \geq 4$ for all $j=1,\dots,s$ and  $n\geq 3$.
% \end{enumerate}
% Then $\Folx{X}{1}{1}$ has exactly two irreducible components.
% \end{MainTheorem}
% }

{
\begin{MainTheorem} \label{T: Main components}
Let $X$ be a smooth complete intersection in $\p{n+s}$ of type $(d_1,\ldots,d_s)$, such that either
\begin{enumerate}
    \item $X$ is a hypersurface, i.e. $s=1$, and $(n,\deg(X)) \neq (3,2)$; or
    \item $d_j \geq 3$ for all $j=1,\dots,s$, $n\geq 4$, and $c_1(TX)\geq 2$; or
    \item $d_j \geq 4$ for all $j=1,\dots,s$ and  $n\geq 3$.
\end{enumerate}
Then $\Folx{X}{1}{1}$ has exactly two irreducible components of logarithmic type. Moreover, every such foliation is the restriction of a foliation on $\p{n+s}$.
\end{MainTheorem}

%The first author proved (see \cite[Teorema 2.2.5]{TeseMateus}) that codimension-one foliations of degree $d$ on a smooth complete intersection (see Remark \ref{R:strong}) extend to the ambient projective space whenever $2d+1 < d_j$ for every $j$. Making $d=1$, we get item (3) and part of (1) of Theorem \ref{T: Main components} from Jouanolou's Theorem. 

To prove Theorem \ref{T: Main components}, we first establish Theorem \ref{T:Extdeg1general}, which improves \cite[Teorema 2.2.5]{TeseMateus}. Then we focus on varieties covered by lines. We study the deformations of lines along the leaves of a foliation, as introduced in \cite{LPT-def}. First, we show that a general free line is not invariant under any foliation (see Proposition \ref{P:line-notFinvar}). A detailed analysis of the \emph{tangential foliation} leads to the following structure theorem for foliations of small degrees on varieties covered by lines (see Theorem \ref{T:general-cov-by-lines}):

\begin{MainTheorem}
Let $X\subset \p{N}$ be a smooth projective variety of dimension $n\geq 3$ with $\Pic(X)=\Z$, covered by lines. If $\F$ is a codimension-one  holomorphic foliation on $X$ of degree $d$ such that either $d=1$, $X$ is a quadric and $d\leq 2$, or $X = \p{n}$ and $d\leq 3$. Then one of the following holds:
\begin{enumerate}
    \item $\F$ is defined by a closed rational 1-form without codimension-one zeros;
    \item $\F$ is algebraically integrable;
    \item $\F$ admits an algebraically integrable subfoliation of codimension two and degree $k$, with $k \leq d - 2$.
\end{enumerate}
\end{MainTheorem}

Specializing to degree-one foliations on smooth complete intersections with $c_1(TX)\geq 2$, we conclude that either the foliation is defined by a closed rational 1-form without codimension-one zeros, or it admits a codimension-two subfoliation of degree zero (see Corollary \ref{C:deg one}). In the latter case, we prove that the foliation is the pullback of a degree-one foliation on the projective plane, and is therefore again defined by a closed rational 1-form without codimension-one zeros (see Theorem \ref{T:degreoneSCI}). The last item follows from Theorem \ref{T:Extdeg1general}. This completes the proof of the last two items of Theorem \ref{T: Main components} and of the first for smooth hypersurfaces either of degree $\deg(X) \leq n$ and dimension $n\geq 4$; or degree $\deg(X)\geq 4$ and dimension $n\geq 3$. The only missing case is $\deg(X) \leq 3$ and $n=3$. For $\deg(X) =1$ the result is immediate, and for $\deg(X) = 2$ the conclusion is false, as proven in \cite{LPT13}. The cubic case follows from Theorem \ref{T:ext-deg1-cub3fold}.

}

This article is organized as follows. In Section \ref{Sec-Fol}, we introduce the main definitions and background on holomorphic foliations, and describe the space of codimension-one foliations on projective varieties. Section \ref{Sec-ext} is devoted to the study of extensions and restrictions of foliations. In Section \ref{Sec-structure}, we develop the tangential foliation by analyzing deformations of rational curves along the leaves, and use this to prove structural results for codimension-one foliations of degrees one and two on complete intersections. In Section \ref{Sec-deg1comp}, we apply the previous results to complete the classification of degree-one foliations on complete intersections.

\section{Holomorphic foliations}\label{Sec-Fol}
Let $X$ be a smooth projective complex manifold. A \emph{distribution} $\F$ on $X$ is the data of a saturated subsheaf $\TF\subset \TX$, often called the \emph{tangent sheaf} of $\F$. The quotient $\nF = TX/\TF$ is a torsion-free sheaf called the \emph{normal sheaf}. The \emph{singular locus} $\sing(\F)$ is the locus where $\nF$ is not locally free; since $\nF$ is torsion free, $\codim \sing(\F) \geq 2$. {These sheaves fit in a short exact sequence:
\[
0 \lra \TF \lra \TX \lra \nF \lra 0.
\]}
The tangent sheaf $\TF$ is \emph{involutive} if it is closed under the Lie bracket of vector fields; in this case, $\F$ defines a \emph{foliation}, i.e., through every point $x \in X\setminus \sing(\F)$ passes an immersed analytic submanifold $\mathcal{L}$ of dimension $\rk \TF$ such that $T_x\mathcal{L} = T\F_x$; these are the \emph{leaves} of the foliation.

The map $\TX \to \nF$ is defined by contraction with a twisted $p$-form $\omega_\F \in H^0(X,\Omega^p_X\otimes \det(\nF))$, where $p = \rk \nF$. Conversely, given a line bundle $L\in \Pic(X)$,  $\omega \in H^0(X,\Omega^p_X\otimes L)$
defines a distribution if and only if it is \emph{locally decomposable off the singular set}, LDS for short, see \cite{Med}. This condition means that for every point $x\in X\setminus \sing(\omega)$ there exists an open neighborhood $U\ni x$ and $1$-forms $\alpha_1, \dots, \alpha_p$ such that 
\[
\omega|_U = \alpha_1 \wedge \cdots \wedge \alpha_p.
\]
Moreover, $\omega$ is integrable (i.e., defines a foliation) if it is LDS and satisfies, for every $j = 1, \dots, p$,
\[
d\alpha_j \wedge \alpha_1 \wedge \cdots \wedge \alpha_p=0.
\]
The local decomposability and integrability condition determine algebraic equations on $\p{}H^0(X,\Omega^p_X\otimes L)$ defining the space of foliations
\[
\Folx{X}{p}{L} \subset \p{}H^0(X,\Omega^p_X\otimes L).
\]

The main objective of our work is to describe the irreducible components of this space for $p=1$, $X\subset \p{N}$ a smooth complete intersection, and specified $L$. 
Any line bundle on $\pn$ is isomorphic to $\opn(k)$ for some $k\in \ZZ$. A distribution has \emph{degree} $d$ if it is defined by $\omega \in H^0(\omn{p}(d+p+1))$. This follows Lins Neto's convention of the degree of a foliation. In general, if $X$ is a smooth projective variety such that $\Pic(X)=\ZZ$, we define the \emph{degree of a foliation $\F$} determined by a section in $ H^0(X,\Omega_X^p(d+p+1))$ as $\deg(\F):=d$.
In this case, for simplicity, we will denote $\Folx{X}{1}{d}:=\Folx{X}{1}{\mathcal{O}_{X}(d+2)}$. We are particularly interested in smooth complete intersections.

Smooth complete intersections of dimension at least three in $\pn$ have Picard group isomorphic to $\ZZ$. This is due, essentially, to the Lefschetz Hyperplane Theorem and the following fact. 

\begin{lemma}\label{L:complete-intersection}
Let $X\subset \p{N}$ be a smooth complete intersection of codimension $k$. Then there exist hypersurfaces $D_1, \dots, D_k \subset \p{N}$ such that $X_i = D_1 \cap \dots \cap D_i$ is smooth for $i = 1, \dots, k$ and $X=X_k$.
\end{lemma}

\begin{proof}
Assume $X = V(f_1, \dots, f_k)$ with $\deg(f_1) \geq \deg(f_2) \geq \cdots \geq \deg(f_k)$; denote $d_i = \deg(f_i)$. For each $1\leq i < j \leq k$, let $h_{ij}$ be a polynomial of degree $d_i - d_j$, to be chosen generically. Then define 
\[
g_i = f_i + \sum_{j=i+1}^k h_{ij}f_j \quad \text{and} \quad X_i = V(g_1, \dots, g_i);
\]
clearly $X  = X_k$. We need to show that $X_i$ is smooth for a generic choice of $h_{ij}$. We proceed by induction. For the sake of the argument, define $X_0 = \pn$. Suppose $X_{i-1}$ is smooth, for some $i\geq 1$.

Note that $X_i = X_{i-1} \cap V(g_i)$ is an element of the linear system  $\Lambda $ spanned by $f_i$ and $f_jh_{ij}$, for $h_{ij}$ varying in a basis of $H^0(\OO_{X_{i-1}}(d_i-d_j))$. By Bertini's Theorem, $X_i$ can be chosen smooth away from the base locus, which is precisely $X$. Now, if $X_i$ were singular at some point $x\in X$, then $X = X_i \cap V(f_{i+1}, \dots, f_k)$ would also be singular at $x$, which is absurd. 
\end{proof}

\begin{remark}\label{R:strong}
In {\cite{ACM-dist,TeseMateus}} the authors study distributions and foliations on smooth complete intersections satisfying an extra hypothesis: $X=V(f_1, \dots, f_k)$ with $\deg(f_1) \leq \dots \leq \deg(f_k)$ such that $X_i = V(f_1, \dots, f_i)$ is smooth for every $i$. There are smooth complete intersections without this property; for instance, $X = V(f_1,f_2)$ where $\deg(f_1) < \deg(f_2)$ and $V(f_1)$ is singular. However, one can verify that the arguments in these papers depend uniquely on the existence of a chain of smooth complete intersections $X_1 \supset \cdots \supset X_k = X$, regardless of the degrees. Therefore, in light of Lemma \ref{L:complete-intersection}, their results apply to smooth complete intersections in general. 
%{ Remark 2.1: is the citation [15] correct here? should it be [14]? M: Retirei, o Sebastian (rs) estava certo, não tem no artigo.}
\end{remark}

\begin{remark}\label{deg-posit}
Note further that a smooth complete intersection only admits distributions and foliations of nonnegative degree. Indeed, $H^0(X,\Omega_X^p(d+p+1)) = 0$ for $d<0$ and $X$ a smooth complete intersection, see \cite[Lemma 5.17]{ACM-dist}.
\end{remark}

\subsection{Algebraically integrable foliations} 
A foliation on $X$ is \emph{algebraically integrable} if the leaf through a general point is an algebraic subvariety.
In this case, there exists a dominant rational map $\pi \colon X \dashrightarrow Y$ onto a projective variety $Y$ such that the closure of a leaf is an irreducible component of a fiber of $\pi$, see \cite[\S3]{AD-fano}.

\begin{proposition}\label{P:rat-int-or-sub}
Let $X$ be a complex projective manifold of dimension greater than $2$ and with $\Pic(X) \cong \ZZ$. Let $\F$ be a codimension-one foliation on $X$. Suppose that $\F$ is algebraically integrable.  Then either:
\begin{enumerate}
    \item     $\F$ admits a rational first integral with only irreducible fibers; or
    \item     there is a codimension-two subfoliation $\G\subset \F$ such that $c_1(\nF) = c_1(\nG) + \Delta$, with $\Delta \geq 0$.
\end{enumerate}
%Moreover, if $\TF$ is semistable, then $c_1(\TF) \leq 0$. 
\end{proposition}

\begin{proof}
This follows directly from the arguments in \cite[Lemma 3.1 and Theorem 3.2]{LPT13}.
\end{proof}

\begin{corollary}\label{C:rat-int-sci}
Let $X\subset \p{N}$ be a smooth complete intersection of dimension $n\geq3$ and $\F$ be a codimension-one holomorphic foliation on $X$ of degree $d$. 
If $\F$ is algebraically integrable, then
    \begin{enumerate}
         \item $\F$ is defined by a rational $1$-form without zeros in codimension-one; or 
    \item $\F$ admits a codimension-two subfoliation $\G$ of degree $k$ such that $k \leq d-1$.        
    \end{enumerate}   
\end{corollary}

\begin{proof}
By Proposition \ref{P:rat-int-or-sub}, either $\F$ admits a rational first integral $\rho\colon X \dashrightarrow \p1$ with only irreducible fibers or admits a subfoliation $\G$ {such that $c_1(\nG) \leq c_1(\nF)$; hence $\deg(\G) \leq d-1$.} Suppose we are in the first case. By Lefschetz Hyperplane Theorem, $X$ is simply connected. Then the argument follows the proof of \cite[Proposition 3.5]{LPT13}. By \cite[Theorem 3.3]{LPT13}, $\rho$ admits at most $2$ multiple fibers, which we may assume to be over $0$ and $\infty$, and write $\rho^{-1}(0) = pH_0$ and $\rho^{-1}(\infty) = qH_\infty$. Then $\rho^*(\frac{dx}{x} - \frac{dy}{y})$ is a logarithmic $1$-form defining $\F$ with polar divisor $H_0+H_\infty$ and empty zero divisor. 
\end{proof}

Notice that the proof of Corollary \ref{C:rat-int-sci} holds for any simply connected manifold $X$ with Picard number $1$. For our purposes, we will only need $X$ to be a smooth complete intersection, especially hypersurfaces.

\subsection{Irreducible components of the space of foliations}\label{S:components}
A degree-zero codimension-one foliation on $\pn$ is determined by a twisted $1$-form $\omega \in H^0(\pn,\omn{1}(2))$. Any of them can be written as
\[
\omega=L_1dL_2-L_2dL_1,
\]
where $L_1$ and $L_2$ are degree one homogeneous polynomials, and thus they are all conjugate to each other (see \cite[Proposition 3.1]{Deserti-Cerveau}). Therefore, $\Folx{\pn}{1}{0}$ is isomorphic to the Grassmannian of lines in $\pn$.

For codimension-one foliations of degree one, Jouanoulou in \cite{JP-Jouanoulou79} proved that  $\Folx{\pn}{1}{1}$ has two irreducible components. Both components parameterize logarithmic foliations, which we now describe in more generality.

Let $X$ be any smooth projective complex variety such that $\dim(X)\geq 3$, $\Pic(X)=\ZZ$ and $H^1(X,\CC)=0$. Given $d_1,d_2,\dots, d_k$ positive integers and $\lambda_1,\lambda_2,\dots,\lambda_k$ complex numbers, not all zero, such that $\sum \lambda_i\cdot d_i=0$, we can take sections $f_i\in H^0(X, \OX(d_i))$ to form global sections
\begin{equation}\label{log-forms}
\omega = \left(\prod^k_{i=1}f_i\right)\sum^k_{i=1}\lambda_i\dfrac{df_i}{f_i} \in H^0(X, \Omega^1_{X}(\sum d_i)).
\end{equation}
Then $\omega$ is integrable and, if its singular locus has codimension greater than one, it determines a foliation known as a \emph{logarithmic foliation}. Varying $\lambda_i$ and $f_i$, $i = 1, \dots, k$, the forms $\omega$ sweep an irreducible algebraic variety subvariety of $\Folx{X}{1}{d_1 + \dots + d_k-2}$ whose closure is denoted by $\Log(X;d_1,d_2,\dots,d_k)$. In \cite{Calvo94},  it was established that $\Log(X;d_1,d_2,\dots,d_k)$ is an irreducible component of $\Folx{X}{1}{\sum d_i-2}$.    {These components are called } \emph{logarithmic components}.  {In the special case $k=2$, the foliation defined by the 1-form 
\[
\omega= \lambda_1f_2df_1+\lambda_2f_1df_2, 
\]
is algebraically integrable, and the corresponding irreducible components are also called \emph{rational components}}. These varieties can also be defined by the closure of the {image of } rational map
\[
\begin{array}{crcl}
\Phi_X\colon &\Sigma\times \displaystyle\prod^k_{i=1}\p{} H^0(X,\OX(d_i))&\dashrightarrow & \Pj H^0(X,\Omega^1_X(\sum d_i))\\
&((\lambda_1:\dots :\lambda_k),f_1,\dots ,f_k)&\mapsto & \left(\displaystyle\prod^k_{i=1}f_i\right)\left(\displaystyle\sum^k_{i=1}\lambda_i\dfrac{df_i}{f_i}\right)\end{array},
\]
where $\Sigma\subset\p{k-1}$ is the hyperplane defined by $\sum\lambda_i\cdot d_i=0$, if its domain is not empty.

{The irreducible components of $\Folx{\pn}{1}{1}$ are the rational component $\Log(\pn;2,1)$ and the logarithmic component $\Log(\pn;1,1,1)$.} We can also describe $\Log(\pn;1,1,1)$ as linear pullbacks of degree-one foliations on $\p2$. In general, {  let  $\LPB(d)$ denote the closure of}  the image of the map
\[
\begin{array}{rcl}
 \Theta\colon \mathbb{L}_2\times \Folx{\p{2}}{1}{d}&\longrightarrow  & \Folx{\pn}{1}{d}  \\
    (\pi,\G) &\longmapsto &\pi^*\G 
\end{array},
\]
where $\mathbb{L}_2$ is the set of linear projections of $\p{n}$ to $\p{2}$. Then, in \cite{CL-NE}, it was demonstrated that $\LPB(d)$ is an irreducible component of $\Folx{\pn}{1}{d}$, and it is called the \emph{linear pullback component}. It is an elementary exercise to show that $\LPB(1)=\Log(\pn;1,1,1)$. 

In degree two, Cerveau and Lins Neto proved in \cite{LNC-96} that $\Folx{\pn}{1}{2}$ has six irreducible components: a linear pullback component $\LPB(2)$; the rational components $\Log(\pn;3,1)$ and  $\Log(\pn;2,2)$; the logarithmic components $\Log(\pn;2,1,1)$ and $\Log(\pn;1,1,1,1)$; and another irreducible component called the \emph{exceptional component}, which is determined by the closure of the map
\[
\begin{array}{rcl}
 \Psi\colon \mathbb{L}_3&\rightarrow  & \Folx{\p{n}}{1}{2}  \\
    \pi &\mapsto &\pi^*\mathcal{E}
\end{array},
\]
where { $\mathbb{L}_3$ is the set of linear projections of $\p{n}$ to $\p{3}$  and} $\mathcal{E}$ is the \emph{exceptional foliation} determined by the section
\[
\beta = \frac{1}{x_0}(2fdg-3gdf) \in H^0(\p{3},\Omega^1_{\p{3}}(4)),
\]
for $f= 2x_0x_2-x_1^2$ and $g= 3x_0^2x_3 - 3x_0x_1x_2 + x_1^3$.

In \cite{CLP-2022}, the irreducible components of $\Folx{\pn}{1}{3}$ were partially classified; there are at least 24 irreducible components. Besides the projective space, there is some literature on irreducible components of the space of foliations on, for instance, Fano three-folds \cite{LPT13}, homogeneous spaces \cite{BFM,kuster2025}, and hypersurfaces \cite{ACM-dist,Fig23}.

In the case of smooth projective hypersurfaces $X$ in $\p{n+1}$, $n>3$, if 
\[
\deg(X)>2\deg(\F)+1,
\]
then the restriction of foliations from the ambient $\p{n+1}$ defines an isomorphism between $\Folx{X}{1}{d}$ and $\Folx{\p{n+1}}{1}{d}$, see \cite[Theorem 3.6]{Fig23}. Therefore, the previous classifications for foliations on projective spaces are valid for these varieties. The problem is then to study low-degree hypersurfaces; for degree-one foliations, we need to analyze hypersurfaces of degree at most three.

In \cite{LPT13}, the authors proved that for $X\subset \p4$ a smooth quadric hypersurface, $\Folx{X}{1}{1}$ has three irreducible components; only two are restrictions from foliations on $\p4$. In this work, we prove that, for hypersurfaces, this is the only case where $\Folx{X}{1}{1}$ has more than two irreducible components. 

% {These last 3 paragraphs of 2.2 give a good overview on your
% contribution in the case of hypersurfaces and should probably go in the intro
% duction. Also, in order to give a broader scope of the literature you may also
% want to reference the works constructing irreducible components of spaces of
% foliations on toric varieties (linear pullbacks) and on projective spaces in posi
% tive characteristic (logarithmic foliations). }
\subsection{Degree-zero distributions on the projective space}
In \cite{ACM-dist}, the authors classify distributions of codimension one and degree zero on $\pn$. There is a stratification according to the \emph{class} of the $1$-form defining the distribution. The argument boils down to the classification of skew-symmetric matrices and the nesting of secant varieties to the Grassmannian of lines in $\p{}(\extp^2\CC^{n+1})$. In higher codimension, the situation is surprisingly simpler.

\begin{proposition}\label{P:deg0-integrable}
Let $\D$ be a degree-zero distribution of codimension $q \geq 2$ on $\pn$. Then $\D$ is integrable.
\end{proposition}

\begin{proof}
Let $\omega \in H^0(\p{n},\omn{q}(q+1))$  be a $q$-form defining $\D$. By \cite[Theorem A]{Med}, there exists a choice of homogeneous coordinates such that either $\omega = \alpha\wedge dx_0\wedge \cdots \wedge dx_{q-2}$ for some 1-form $\alpha$ with linear coefficients, or $\omega = i_v dx_0\wedge \cdots \wedge dx_{q}$, for some linear vector field $v$. Imposing $i_\rad \omega = 0$, for $\rad = \sum_j x_j \del{ x_j}$ the radial vector field, we must have $\omega = i_\rad dx_0\wedge \cdots \wedge dx_{q}$, up to a scalar multiple. Therefore, $\D$ is integrable. 
\end{proof}

\section{Extensions and restrictions of foliations on complete intersections}\label{Sec-ext}
Let $U_k$ be the open subset of twisted $q$-forms in $ H^0(\p{N},\om{q}{N}(k))$ whose singular set has codimension greater than or equal to two, and let  $X\subset \p{N}$ be a smooth complete intersection.

\begin{definition}\label{D:trans}
We say that $\omega\in U_{k}$ is transversal in codimension one to $X$ if $i^*\omega$ has a singular set of codimension greater than or equal to two, where $i\colon X\rightarrow \p{N}$ is the natural inclusion. A singular holomorphic foliation $\F$  on $\p{N}$ of codimension $q$ is \emph{transversal in codimension one} to $X$, if a twisted $q$-form that determines it is transversal in codimension one to $X$.
\end{definition}

If a foliation $\F$ on $\p{N}$ is transversal in codimension one to $X$, the restriction to $X$ of its twisted $q$-form naturally satisfies the decomposability and integrability hypotheses and, therefore, defines a foliation $\G$ on $X$.

\begin{definition}
The foliation $\G$ will be called the \emph{restriction} of $\F$ to $X$. We will denote $\restr{\F}{X}:=\G$.
If a foliation $\G$ on a smooth complete intersection $X\subset \p{N}$ is the restriction of a foliation $\F$ in $\p{N}$ to $X$, we say that $\F$ is an \emph{extension} for $\G$.
\end{definition}

\subsection{Codimension-one foliations}
Let $X\subset \p{N}$ be a smooth complete intersection of dimension $r\geq 3$. The restriction maps
\[
\rest_0 \colon H^0(\p{N},\mathcal{O}_{\p{N}}(k))\rightarrow H^0(X,\OX(k))
\]
are surjective, for all $k\in \mathbb{Z}$. %(see \cite[Chapter III, Exercise 5.5]{hartshorne}). 
Thus, if $\F$ is an algebraically integrable foliation on $X$ defined by a pencil of hypersurfaces sections  $f, g $, where $f\in H^0(X,\OX(\deg(f))$ and $g\in H^0(X,\OX(\deg(g))$. Then, the 1-form $\deg(g)gdf-\deg(f)fdg$ determines $\F$. Therefore, we can find coprime polynomials $F,G$ such that $\rest_0(F)=f$ and $\rest_0(G)=g$, and the foliation restriction generated by $\deg(g)GdF-\deg(f)FdG$ is {an} extension of $\F$.

In general, for logarithmic 1-forms as in \eqref{log-forms}, we have the commutative diagram 
\begin{equation}\label{E:logmap}
\xymatrix{
\Sigma\times \left(\prod^k_{i=1}\Pj H^0(\p{N},\mathcal{O}_{\p{N}}(d_i))\right)\ar@{-->}[r]^{\,\,\,\,\,\,\,\,\,\,\,\,\,\,\,\,\Phi_{\p{N}}} \ar@{-->}[d]_{\text{id}\times \rest_0} & \Pj H^0(\p{N},\om{1}{N}(\sum d_i)) \ar@{-->}[d]^{i^*} \\
\Sigma\times \left(\prod^k_{i=1}\Pj H^0(X,\OX(d_i))\right) \ar@{-->}[r]_{\,\,\,\,\,\,\,\,\,\,\,\,\,\,\,\,\,\,\,\,\Phi_{X}} & \Pj H^0(X,\omx{1}(\sum d_i)) }
\end{equation}
which projects all logarithmic 1-forms of $\p{N}$ onto $X$ surjectively, where $i \colon X\rightarrow \p{N}$ is the natural inclusion map. Therefore, all foliations generated by logarithmic 1-forms extend.

\begin{proposition}\label{P:logext}
    Let $X$ be a smooth complete intersection. Then, every codimension-one foliation on any logarithmic component $\Log(X,d_1,d_2,\dots,d_k)$ extends.
\end{proposition}

\begin{proof}
    As seen above, if $\F$ is generated by a logarithmic 1-form, it has an extension. Furthermore, the logarithmic and rational components can be determined by the closure of the image of a rational map as \eqref{E:logmap}.

    Suppose $\F$ is in a logarithmic component $\Log(X,d_1,d_2,\dots,d_k)\subset\Folx{X}{1}{\sum d_i -2}$ and is an adherence value of the map $\Phi$. In this case, there exists a sequence of foliations $\{\F_i\}_{i=0}^{\infty}$ in the logarithmic component such that $\F_i\rightarrow \F$ and each $\F_i$ is defined by a logarithmic 1-form. Since each foliation $\F_i$ extends, there exists a sequence of foliations $\{\G_i\}_{i=0}^{\infty}$ in the logarithmic component $\Log(\p{N},d_1,d_2,\dots,d_k)$, where $\restr{\G_i}{X}=\F_i$ and each $\G_i$ is determined by a logarithmic 1-form $\omega_i$.
    Since the integrability condition determines a proper closed set in the class space $\Pj H^0(\p{N},\om{1}{N}(k)),$ the set of points adhering to the sequence $([\omega_i])_{i=0}^{\infty}$ is nonempty and each point in this set is integrable. Thus, given $[\omega]$ as one of these adherent points, its restriction to $X$ determines $\F$ by definition. Furthermore, we have that $\codim(\omega)\geq 2$, otherwise, $\codim(\rest(\omega))\leq 1$ and $\rest(\omega)$ would not determine $\F$. Therefore, $\omega$ determines a logarithmic foliation $\G$ on $\p{N}$ whose restriction is $\F$.
\end{proof}

{
In general, foliations on a complete intersection $X$ extend to the projective space when the degree of the foliation is sufficiently small compared to the degrees of the hypersurfaces defining the complete intersection. In the case of hypersurfaces, \cite[Theorem 3.6]{Fig23} shows that a foliation $\F$ has an unique extension if $\deg(X) > 2\deg(\F)+1$. For degree-one foliations on complete intersections, we have the following result, cf. \cite[Teorema 2.1.1]{TeseMateus}. 

% \begin{theorem}\label{T:Extdeg1general}
% Let $X$ be a smooth complete intersection in $\p{n+s}$ of type $(d_1,\ldots,d_s)$ and dimension $n\geq 3$.
% If $d_j\geq 4$ for $j  = 1, \dots, s$, then any codimension-one foliation on $X$ of degree one has a unique extension to $\p{n+s}$.
% \end{theorem}

% \begin{proof}
% Let $\omega\in H^0(X, \Omega^1_X(\deg(3))$ be a section determining $\F$ on $X$. In light of Remark \ref{R:strong}, \cite[Proposition 5.22]{ACM-dist} holds for smooth complete intersections as in the hypotheses above. Thus, the restriction map
% \[
% r_1 \colon H^0(\p{n+s}, \Omega^1_{\p{n+s}}(3))\rightarrow H^0(X, \Omega^1_X(3))
% \]
% is an isomorphism, and 
% \[
% r_3 \colon H^0(\p{n+s}, \Omega^3_{\p{n+s}}(6))\rightarrow H^0(X, \Omega^3_{X}(6))
% \]
% is injective. Therefore, there exists a unique section $\alpha\in H^0(\p{n+s}, \Omega^1_{\p{n+s}}(3))$ whose restriction to $X$ is $\omega$. Note that $\alpha$ must have singularities in codimension at least two, as so does $\omega$. Observe that  
% \[
% r_3(\alpha\wedge d\alpha)= \omega \wedge d\omega = 0.
% \]
% Then $\alpha$ is integrable by the injectivity of $r_3$.
% \end{proof}

{
\begin{theorem}\label{T:Extdeg1general}
Let $X$ be a smooth complete intersection in $\p{n+s}$ of type $(d_1,\ldots,d_s)$ and dimension $n\geq 3$. If $d_j\geq 4$ for $j  = 1, \dots, s$, then the restriction induces an isomorphism
\[
\Folx{X}{1}{1} \cong \Folx{\pn}{1}{1} .
\]
\end{theorem}

\begin{proof}
Let $\omega\in H^0(X, \Omega^1_X(3))$ be a section determining $\F$ on $X$. In light of Remark \ref{R:strong}, \cite[Proposition 5.22]{ACM-dist} holds for smooth complete intersections as in the hypotheses above. Thus, the restriction map
\[
r_1 \colon H^0(\p{n+s}, \Omega^1_{\p{n+s}}(3))\rightarrow H^0(X, \Omega^1_X(3))
\]
is an isomorphism, and 
\[
r_3 \colon H^0(\p{n+s}, \Omega^3_{\p{n+s}}(6))\rightarrow H^0(X, \Omega^3_{X}(6))
\]
is injective. Therefore, we have the following commutative square, see \cite[\S2]{BFM}, 
\[
\begin{tikzcd}
    S^2H^0(\p{n+s}, \Omega^1_{\p{n+s}}(3)) \ar[r, "S^2 r_1","\sim"'] \ar[d,two heads,"\Psi_{\p{n+s}}"] & S^2H^0(X, \Omega^1_X(3))\ar[d,"\Psi_X"] \\
    H^0(\p{n+s}, \Omega^3_{\p{n+s}}(6)) \ar[r,"r_3",hook] & H^0(X, \Omega^3_{X}(6))
\end{tikzcd}
\]
where $\Psi_Y(\omega\cdot \eta) = \omega \wedge d\eta + \eta\wedge d\omega$ for $Y = \p{n+s}$ or $Y = X$. Considering the projective space $\p{}H^0(Y, \Omega^1_{Y}(3))$, the dual of the image of $\Psi_Y$,  
\[
(\im \Psi_Y)^\vee \hookrightarrow S^2H^0(Y, \Omega^1_{Y}(3))^\vee, 
\]
generate the homogeneous ideal of the scheme of integrable forms. From the diagram above, we conclude that $\im \Psi_\pn$ is taken to $\im \Psi_X$ via the isomorphism of projective spaces defined by $r_1$. Hence the schemes ${\rm IF}(\p{n+s},3)$ and ${\rm IF}(X,3)$ of integrable forms are isomorphic. To conclude the isomorphism of the spaces of foliations, note that $r_1(\omega)$ has singularities in codimension one if and only if so does $\omega$. Indeed, if $r(\omega) = f\eta$ and $f$ is not constant, then $\deg(f) = 1$ and $\eta \in H^0(\Omega_X^1(2))$. Since both $f$ and $\eta$ have unique extensions to $\pn$, we get that $\omega$ vanishes along some hyperplane. Then $r_1$ preserves the open subset $U$ of forms without divisoral singularities. We conclude that 
\[
\Folx{\p{n+s}}{1}{1} = U \cap {\rm IF}(\pn,3) \cong \Folx{X}{1}{1}.
\]
\end{proof}
}

An immediate consequence is that the space of degree-one foliations on smooth complete intersections satisfying the hypothesis of the theorem is the same as the one for the projective space, cf. \cite[Corolário 2.2.7]{TeseMateus}.

\begin{corollary}\label{C:ext-comp-int-deg>4}
Let $X$ be a smooth complete intersection in $\p{n+s}$ of type $(d_1,\ldots,d_s)$ and dimension $n\geq 3$.
If $d_j\geq 4$ for $j  = 1, \dots, s$, then 
\[
\Folx{X}{1}{1} = \Log(X;1,2) \cup \Log(X;1,1,1).
\]
\end{corollary}
}

\subsection{Codimension-two foliations}

For codimension-two degree-{zero} foliations, we have the following results regarding extensions.

\begin{theorem}\label{T:Ext2codimgeneral}
Let $X$ be a smooth complete intersection in $\p{n+s}$ of type $(d_1,\ldots,d_s)$ and dimension $n\geq 4$.
If $d_j\geq 3$ for $j  = 1, \dots, s$, then any codimension-two foliation on $X$ of degree zero has a unique extension to $\p{n+s}$.
\end{theorem}

\begin{proof}
Let $\F$ be a codimension-two foliation on $X$ of degree zero. In particular, $\F$ is a codimension-two distribution on $X$ of degree zero. Applying \cite[Corolário 2.3.3]{TeseMateus}, we see that $\F$ has a unique extension as a codimension-two distribution to $\p{n+s}$ of degree zero. Finally, by Proposition \ref{P:deg0-integrable}, we conclude that the extension of $\F$ is integrable.       
\end{proof}

\begin{corollary}\label{C:Ext2codim}
Let $X$ be a smooth hypersurface of $\p{n+1}$, $n\geq 4$. 
If $\F$ is a codimension-two holomorphic foliation on $X$ of degree zero then $\F$ extends {to} $\p{n+1}$.
\end{corollary}

\begin{proof}
 It follows by Theorem \ref{T:Ext2codimgeneral}, if $\deg(X)\geq 3$ and, for quadrics, by \cite[Proposition 3.18]{AD-fano2}.    
\end{proof}

%%%%%%%%%%%%

\section{Structure theorem for foliations on complete intersections}\label{Sec-structure}
In this section, we prove our main technical result, Theorem~\ref{T:general-cov-by-lines}. Then, we derive the structure of degree-one foliations on smooth complete intersections and degree-two foliations on quadric hypersurfaces. We start with some preparation.

\subsection{Tangential foliations}
A smooth projective variety $X$ is called \emph{uniruled} if, for every general point $x \in X$, there exists a non-constant morphism $f\colon \p1 \to X$ such that the rational curve $C=f(\p1)$ passes through $x$ and  $f^*\TX$ is generated by global sections; such morphisms are called free. The scheme $\Mor(\p1, X)$ parametrizing morphisms from $\p1$ to $X$ is smooth around a free morphism $f$, and its tangent space is naturally isomorphic to $H^0(X,f^*\TX)$. We refer to \cite{Deb-book,Kollar-book} for further details. A particular result we will need is the following.

\begin{lemma}\label{L:free-avoidance} Let $X$ be a smooth uniruled projective variety, let $f\colon \p1 \to X$ be a free morphism, and $Z\subset X$ be a closed subvariety of codimension at least two. Then, $f'(\p1) \cap Z = \emptyset$
for a general deformation $f'$ of $f$. 
\end{lemma}

\begin{proof}
    This is a particular case of \cite[II, Proposition 3.7]{Kollar-book}.
\end{proof}

\begin{lemma}\label{L:gen-free}
Let $X$ be a smooth complex quasi-projective variety. There exists a non-empty subset $X^{\rm free}$ of $X$ which is the intersection of countably many dense open subsets of $X$, such that any rational curve on $X$ whose image meets $X^{\rm free}$ is free.
\end{lemma}

\begin{proof}
See \cite[Proposition 4.14]{Deb-book} for the proof of the existence of $X^{\rm free}$ over any algebraically closed field of characteristic zero. { Since $\CC$ is uncountable,} $X^{\rm free}\neq \emptyset$.
\end{proof}

Given a foliation $\F$ on $X$ and an irreducible component $M \subset \Mor(\p1, X)$ containing a free morphism. One defines the \emph{tangential foliation} of $\F$ on $M$, denoted  by $\F_{\tang}$,   setting for  a general free morphism  $f\colon \p1\dashrightarrow X$   the tangent space of $\F_{\tang}$ at $f$ equal to  $H^0(\p1, f^*T_{\F})$,  via the identification of the tangent space of $\mathrm{Mor}(\mathbb P^1,X)$ at the morphism $f$ with $ H^0(\p1, f^*T_X)$ and the inclusion $H^0(\p1, f^*\TF) \subset H^0(\p1,f^*\TX)$.  More precisely, the tangential foliation is defined as follows. Consider the evaluation map 
\begin{align*}
    \mathrm{ev}\colon& M\times \p1 \longrightarrow X\\
    &([f],x)\longmapsto f(x)    
\end{align*}
and the natural projection $\pi\colon  M\times \p1 \longrightarrow M$. {Consider $\mathcal{H}$ the foliation on $M\times \p1$ given by the fibers of the projection to $\p1$. Then $\F_{\tang}$ is defined as the direct image $\pi_*(\mathcal{H}\cap \mathrm{ev}^*\F)$, i.e., the tangent sheaf of $\F_{\tang}$ is the saturation in $T_M$ of the image of the composition
\[
\pi_*T_{(\mathcal{H}\cap \mathrm{ev}^*\F)} \longrightarrow \pi_*T_{M\times \p1} \longrightarrow T_M.
\]
}
%We have an isomorphism $T_M\cong \pi_*\mathrm{ev}^*T_X $ {Does this morphism works for every morphism or only on the subset of M of free morphims, (C: i think it works for all M, but I have to check it tmrw - lazycado)}, and so the inclusion $T_{\F}\subset T_X$ induces a morphism $i\colon \pi_*\mathrm{ev}^*T_\F\hookrightarrow T_M $. We define the tangential foliation of $\F$ on $M$ by the saturation of the image of this morphism $i$ in $T_M$. 
In \cite{LPT-def,kuster2025}, the authors show how the geometry of $\F_{\rm tang}$ impacts the geometry of $\F$. We recall below two of their results. 

\begin{proposition}\label{P:Ftang-algint}
\cite[Proposition 4.10]{LPT-def} Let $X$ be a smooth uniruled projective variety and let $\F$ be a codimension-one
 foliation on $X$. If all leaves of $\F_{\rm tang}$ are algebraic, then there exists an algebraically integrable foliation $\G$ contained in $\F$ such that $H^0(\p1,f^*\TF) = H^0(\p1,f^*\TG)$ for any sufficiently general $f\in M$. In particular, $\F$ is the pullback under a rational map of a foliation on a lower-dimensional manifold. 
\end{proposition}

\begin{theorem}\label{T:Ftang-NOTalgint}
\cite[Theorem C]{LPT-def} Let $X$ be a simply connected uniruled projective manifold and let $\F$ be a
 codimension-one foliation on $X$.
Suppose that the general leaf of $\F_{\rm tang}$ is not algebraic, and a general $[f]\in M$ intersects nontrivially and transversely all the algebraic hypersurfaces invariant by $\F$. Then $\F$ is defined by a closed rational $1$-form without divisorial components in its zero set.
\end{theorem}

\begin{remark} \label{R:subfol}
In Proposition~\ref{P:Ftang-algint}, we observe that $f^*\TG$ contains the nonnegative part of $f^*\TF$. Let $a_1\geq \dots \geq a_r$ and $b_1 \geq \dots \geq b_s$ be integers such that $f^*\TF = \bigoplus_{i=1}^r\op1(a_i)$ and  $f^*\TG = \bigoplus_{i=1}^s\op1(b_i)$. The inclusion $f^*\TG \subset f^*\TF$ implies that $a_i \geq b_i$ for $i=1, \dots, s$. Let $l = \max\{\,j \mid a_j \geq 0 \,\}$ and $m = \max\{\,j \mid b_j \geq 0 \,\}$. The equality $H^0(\p1,f^*\TF) = H^0(\p1,f^*\TG)$ implies that $m=l$, and $a_i = b_i$ for $i \leq m$.
\end{remark}

We will often consider rational curves that are not tangent to the foliation. When $\Pic(X) \cong \ZZ$, this is possible due to the following result, which was communicated to us by Jorge Vit\'orio Pereira.

\begin{lemma}\label{P:exist-non-invar-curves}
Let $X$ be a smooth uniruled complex projective variety such that $\Pic(X) \cong \ZZ$. If $f\colon \p1\to X$ is a free morphism and $\F$ is a foliation on $X$, then a general deformation $f'$ of $f$ is not $\F$-invariant.
\end{lemma}

\begin{proof}
Due to Lemma \ref{L:free-avoidance}, we may choose $f'$ so that $f'(\p1)$ does not intersect $\sing(\F)$. Aiming for a contradiction, suppose $f'$ is tangent to $\F$. Recall Bott's partial connection $\nabla^B \colon \TF \otimes \nF \to \nF$ that is defined as follows. Let $v\in \TF$ and $\overline{w} \in \nF$ be local sections, and $w\in \TX$ a lift of $w$. Then $\nabla^B_v(\overline{w}) = \overline{[v,w]}$. Since $f'$ is tangent to $\F$, we get that $f'^*\nabla^B$ is a holomorphic connection on the vector bundle $f'^*\nF$ on $\p1$. By Weil's Theorem \cite[Theorem 10]{Atiyah-conn}, $f'^*\nF$ must be trivial. Using that $\Pic(X) \cong \ZZ$, we conclude that $c_1(\nF) = 0$, hence $\F$ is given by a global holomorphic $p$-form $\omega \in H^0(X,\Omega^p_X)$, $p=\codim \F$. On the other hand, since $X$ is uniruled with $\Pic(X) \cong \ZZ$, it follows that $X$ is Fano. Hence
$h^0(X,\Omega^p_X) = h^{0,p}(X) = 0$, which is a contradiction.
\end{proof}

Next, we turn to a particular case of uniruled varieties: varieties \emph{covered by lines}.  
Let $ L \in \Pic(X) $ be an ample line bundle. A morphism $ f\colon \p1 \to X $ is called a \emph{line} (with respect to $ L $) if  
\[
\int_{\p1} f^* c_1(L) = 1.
\]  
The usual setting is when $ X \subset \pn $ and $ L = \OO_X(1) $, meaning that $ f(\p1) $ is a genuine line of $ \pn $ contained in $ X $.  
A variety $ X $ is said to be covered by lines if there exists a line passing through a general point of $ X $. It follows from Lemma~\ref{L:gen-free} that every line passing through a point in $ X^{\mathrm{free}} $ must be free. The following corollary is an immediate consequence of {Lemma}~\ref{P:exist-non-invar-curves}.  

\begin{corollary}\label{C:line-not-invar}
Let $X$ be a smooth complex projective variety covered by lines such that $\Pic(X) \cong \ZZ$. Let $\F$ be a foliation on $X$; then a general free line is not $\F$-invariant.
\end{corollary}

Suppose $X\subset \p{N}$ is a complete intersection given by equations of degrees $d_1, \dots, d_s$. If $d_1+\cdots + d_s \leq N-1$ then $X$ is covered by lines, see \cite[Proposition 2.13]{Deb-book}. Then, the previous corollary applies. 

\begin{proposition}\label{P:line-notFinvar}
Let $X\subset \p{N}$ be a smooth complex complete intersection of dimension at least three defined by equations of degrees $d_1, \dots, d_s$ such that $d_1+\cdots + d_s \leq N-1$. For any foliation $\F$ on $X$, a general line is not $\F$-invariant. 
\end{proposition}

\subsection{Foliations of degrees one and two} Following the discussion above, we will prove a structure theorem for foliations of small degrees on smooth projective varieties covered by lines. Then, we will apply it to smooth complete intersections and, in particular, to smooth hypersurfaces. The result below is inspired by the structure theorem for degree-three foliations on $\p3$ by da Costa, Lizarbe, and Pereira, see \cite[Theorem A]{CLP-2022}.

% \begin{theorem}\label{T:general-cov-by-lines}
% Let $X\subset \p{N}$ be a smooth projective variety covered by lines of dimension $n\geq 3$ with $\Pic(X)=\Z$.  If $\F$ is a codimension-one singular holomorphic foliation on $X$ of degree $d$ such that either $d=1$ or %\red{the index satisfies}
% \begin{equation}\label{E:propertyI}
%    % \red{\iota_X} = 
%     c_1(TX) \geq n+d-2,
% \end{equation}
% then, one of the following holds:
%     \begin{enumerate}
%         \item  $\F$ is defined by a closed rational 1-form without codimension-one zeros; or
%         \item $\F$ is algebraically integrable; or
%         \item $\F$ admits an algebraically integrable subfoliation of codimension 2, and degree $k$ such that $k \leq d-2$. 
%     \end{enumerate}
% \end{theorem}

% Note that the inequality $c_1(TX) \geq n+d-2$ for $d\geq 2$ only holds for the quadric hypersurface or $\pn$, due to the Kobayashi-Ochiai Theorem \cite[Theorem 1.1 and Theorem 2.1]{KobOch}. 

{
\begin{theorem}\label{T:general-cov-by-lines}
Let $X\subset \p{N}$ be a smooth projective variety covered by lines of dimension $n\geq 3$ with $\Pic(X)=\Z$.  If $\F$ is a codimension-one singular holomorphic foliation on $X$ of degree $d$ such that either 
\begin{itemize}
    \item $d=1$;
    \item $X$ is a quadric and $d\leq 2$; or
    \item $X = \p{n}$ and $d\leq 3$.
\end{itemize}
then, one of the following holds:
    \begin{enumerate}
        \item  $\F$ is defined by a closed rational 1-form without codimension-one zeros; or
        \item $\F$ is algebraically integrable; or
        \item $\F$ admits an algebraically integrable subfoliation of codimension 2, and degree $k$ such that $k \leq d-2$. 
    \end{enumerate}
\end{theorem}

\begin{proof}
   The hypotheses on $X$ and $d$ are equivalent to either $d=1$ or $c_1(TX) \geq n+d-2$. Indeed, for $d\geq 2$ this is a consequence of the Kobayashi-Ochiai Theorem \cite[Theorem 1.1 and Theorem 2.1]{KobOch}.
    By the adjunction formula, we have 
    \[
    c_1(\TF) = c_1(TX)-(d+2).
    \]
    Let $f \colon \p{1} \to \p{N}$ be the parameterization of a free line on $\p{N}$ contained in $X$. Due to Proposition \ref{P:line-notFinvar}, we can assume that $f$ is generically transverse to $\F$ and does not intersect $\sing(\F)$. Consider the decomposition $f^* \TF=\bigoplus_{i=1}^{n-1}\op{1}(a_i)$.
    Since $f$ is generically transverse to $\F$, $f^*T_\F$ injects into the normal bundle $N = f^*TX/\op1(2)$. By Lemma \ref{L:gen-free}, $N$ is globally generated so every summand is nonegative. On the other hand, $N\hookrightarrow f^*T\p{N}/\op1(2) \cong \op1(1)^{\oplus N-1}$; so that every summand is of degree at most one. Therefore, 
    \[
    f^*TX = \OO_{\p1}(2) \oplus \OO_{\p1}(1)^{a}\oplus \OO_{\p1}^{b}
    \]
    where $a=c_1(TX)-2\geq 0$ and $b=n-c_1(TX)+1\geq 0$; see, for instance,  \cite[IV 2.9]{Kollar-book}. It follows that every summand of the decomposition of $f^*\TF$ satisfies $a_i\leq 1$. From 
    \begin{equation}\label{E:Tgeneral}
        \sum_{i=1}^{n-1}a_i=c_1(TX)-(d+2) = a-d\geq -1,
    \end{equation}
    we have that $h^0(\p{1},f^*\TF)\neq 0$, hence $\F_{\rm tang}$ is not trivial. 
   %  
   % We claim that the following are the only possibilities for %$f^*\TF$:
   %\begin{enumerate}
   % \item $f^*\TF$ is nonnegative in the sense that it is a direct sum of holomorphic line bundles of degree $\geq 0$.
   %\item $f^*\TF\simeq \op{1}(1)^{r}\oplus \op{1}^{\oplus n-2-r}\oplus\op{1}(-c) $, for some nonnegative integer $r$ and $c=1$ or $c=2$. 
   %  \end{enumerate}
   % To see this, note that if $d=1$ then $c_1(T_\F) = a-1$ and $a_j \geq 0$ for $j\leq n-2$ and $a_{n-1} \geq -1$, since we . If $d\geq 2$ then $c_1(T_\F) \geq n-4$. Then, $a_j =1$ for $j\leq n-4$ and $a_{n-3}+a_{n-2}+a_{n-1} \geq 0$. Then a direct computation proves the claim.
 We claim that the following are the only possibilities for $f^*\TF$:
\begin{enumerate}
\item $f^*\TF$ is nonnegative, i.e., it is a direct sum of holomorphic line bundles of nonnegative degree.
\item
\[
f^*\TF\simeq \op{1}(1)^r\oplus
\op{1}^{\oplus n-2-r}\oplus
\op{1}(-c),
\]
for some nonnegative integer $r$, where $c=1$ or $2$.
\end{enumerate}
Indeed, if $d=1$, then $c_1(T\F)=a-1$ and $a_j\ge0$ for $j\le n-2$, while $a_{n-1}\ge-1$. Hence $f^*\TF$ is either nonnegative or has a unique negative summand, namely $\op{1}(-1)$. Assume now that $d\ge2$. Since $c_1(T\F)\ge n-4$ and $a_j\le1$, we must have $a_j=1$ for every $j\le n-4$. Consequently,
\[
a_{n-3}+a_{n-2}+a_{n-1}\ge0.
\]
Using the ordering
\[
a_1\ge\cdots\ge a_{n-1}
\]
together with the above inequality, one readily checks that the only remaining possibilities are those listed above.

    Suppose $\F_{\rm tang}$ is not algebraically integrable and that $\F$ is not algebraically integrable. Since $\Pic(X)\cong \ZZ$, every effective divisor is ample; hence it intersects every curve. Since the line is chosen generically and $X$ has Picard number one, this line intersects non-trivially and transversely each of the finitely many algebraic hypersurfaces invariant by $\F$; hence we can apply \cite[Theorem C]{LPT-def} to show that $\F$ is defined by a closed rational 1-form without codimension-one zeros. This gives our first item. 
    
    Suppose that $\F_{\rm tang}$ is algebraically integrable. By \cite[Proposition 4.10]{LPT-def}, there exists an algebraically integrable subfoliation $\G \subset \F$ such that $H^0(f^*T_\G) = H^0(f^*T_\F)$. Write $f^*T_\G =\bigoplus_{i=1}^{n-1}\op{1}(b_i)$. By Remark \ref{R:subfol}, there is $m >0$ such that $b_i = a_i$ for $i \leq m$ and for $i>m$ both $b_j<0$ and $a_j<0$. It follows that either $m = n-1$ and $\G = \F$ is algebraically integrable, or $f^*\TF\simeq \op{1}(1)^{r}\oplus \op{1}^{\oplus n-2-r}\oplus\op{1}(-c)$ and $f^*\TG=\op{1}(1)^{r}\oplus \op{1}^{\oplus n-2-r}$. By the adjunction formula for  $\F$ and $\G$ we have
    \[
    \mbox{ $c_1(\TF)=r-c=c_1(TX)-(d+2)$, and 
    $c_1(\TG) =r= c_1(TX)-(k+3)$.} 
    \]
    Combining and simplifying, we get $k=d-1-c$. This concludes the proof.   
\end{proof}
}

The previous theorem allows us to determine the structure of degree-one foliations on smooth complete intersections.

\begin{corollary}\label{C:deg one}
Let $X$ be a smooth complete intersection on $\p{N}$ of dimension $n\geq3$. Suppose $c_1(TX)\geq 2$. 
 If $\F$ is a codimension-one foliation on $X$ of degree one, then
    \begin{enumerate}
         \item  $\F$ is defined by a closed rational 1-form without codimension-one zeros; or 
          \item $\F$ admits a subfoliation of codimension 2, and degree zero.       
    \end{enumerate}  
\end{corollary}

\begin{proof}

Let $X \subset \p{N}$ be a smooth complete intersection defined by equations of degrees 
$d_1,\dots,d_s$. Since, by hypothesis, $c_1(TX) \ge 2$, we have 
$d_1 + \cdots + d_s \le N-1$ and then $X$ is covered by lines, see \cite[Proposition 2.13]{Deb-book}. 
By Theorem \ref{T:general-cov-by-lines}, either $\F$ is defined by a closed rational 1-form without codimension-one zeros, is algebraically integrable, or admits a codimension-two subfoliation of negative degree. The latter is not possible, as observed in Remark \ref{deg-posit}. Assume now that $\F$ is not given by a closed rational 1-form without codimension-one zeros. Then, by Corollary \ref{C:rat-int-sci}, $\F$ must admit a codimension-two subfoliation of degree zero.
\end{proof}

\begin{corollary}\label{C:new deg_two_quadric}
Let $\F$ be a codimension-one holomorphic foliation of degree two on a smooth quadric $\qn \subset \p{n+1}$, $n\geq 4$. Then
\begin{enumerate}
        \item  $\F$ is defined by a closed rational 1-form without codimension-one zeros; or
         \item $\F$ is a linear pull-back of a degree 2 foliation on $\p2$; or
          \item $\F$ admits a subfoliation of codimension 2, and degree one.    
\end{enumerate}
Moreover, foliations in item (3) are algebraically integrable.
\end{corollary}

\begin{proof}
Suppose $\F$ is not defined by a closed rational 1-form without codimension-one zeros. Combining Theorem \ref{T:general-cov-by-lines} and Corollary \ref{C:rat-int-sci}, we have that $\F$ admits a subfoliation $\G$ of codimension two and degree $k\leq 1$. In case $k=0$, by \cite[Proposition 3.18]{AD-fano2}, $\G$ must be given by a linear projection to $\p2$, hence $\F$ is a linear pullback of a foliation of degree two on $\p2$.
\end{proof}

\begin{theorem}\label{T:degreotwo_3quadric}
Let $\F$ be a codimension-one holomorphic foliation of degree two on a smooth tridimensional quadric $\q{3}$. Then
    \begin{enumerate}
        \item  $\F$ is defined by a closed rational 1-form without codimension-one zeros; or
          \item $\F$ is tangent to a foliation by curves of degree one; or
          \item there exists an algebraic action of $\C$ or $\C^*$ tangent to $\F$ such that when this action has only isolated fixed points, we have $\TF = \mc O_{\q{3}} \oplus \mc O_{\q{3}}(-1)$.
    \end{enumerate}
Moreover, foliations in item (2) are algebraically integrable.
\end{theorem}

\begin{proof}
The proof is similar to the proof of Corollary \ref{C:new deg_two_quadric}. Except when the subfoliation $\G$ has dimension one, and \cite[Proposition 3.18]{AD-fano2} does not apply. Instead, $\TG = \oq3$ and $\G$ is defined by a global vector field on $\q3$ without zeros in codimension one.

Let $\omega\in H^0(\q3 , \Omega_{\q3}^1({4}))$ be a 1-form defining $\F$ and define 
\[
\aut(\F) = \{ v\in H^0(\q3, T\q3) \mid \mathcal{L}_v\omega \wedge \omega = 0 \},
\]
the Lie algebra of infinitesimal automorphisms of $\F$, here $\mathcal{L}_v\omega$ is the Lie derivative. Moreover, define $\fix(\F) = H^0(\q3, T\F) \subset \aut(\F)$. Note that $\G$ is defined by a vector field $v\in \fix(\F)$.

By \cite[Lemma 4.3]{LPT13}, either $\fix(\F) \neq \aut(\F)$ and $\F$ is defined by a closed rational 1-form without zeros in codimension one; or $\fix(\F) = \aut(\F)$ and there exists a non-trivial algebraic action of a one-dimensional Lie group tangent to $\F$; on $\q3$ this group must be $\C$ or $\C^*$. 
Suppose we are in the latter case, and let $\G$  denote the foliation defined by this algebraic action. Finally, if $\G$ has only zero-dimensional singularities, then \cite[Lemma 4.2]{LPT13} implies that $\TF = \oq{3} \oplus \oq{3}(-1)$.
\end{proof}

\section{Irreducible components of
degree-one foliations}\label{Sec-deg1comp}

In this section, we apply the structure Theorem \ref{T:general-cov-by-lines} and its corollaries to determine the irreducible components of $\Folx{X}{1}{1}$ for $X$ either a hypersurface or a smooth complete intersection of hypersurfaces of degrees at least three, improving Corollary \ref{C:ext-comp-int-deg>4}. We start with a technical result.

\begin{lemma}\label{L:crit-proj}
    Let $\pi \colon \p{n+s}\dasharrow \p{q}$ be a linear projection, with $n>3$ and $q<n$, and let $X\subset \p{n +s}$ be a smooth complete intersection {such that $\restr{\pi}{X}$ is a dominant rational map. Then %the restriction $\pi|_{X}$ is a surjective map and 
    the set of its critical points has codimension greater than one.}
\end{lemma}

\begin{proof}
     Let $\Theta \in H^0(\p{q}, \Omega^q_{\p{q}}(q+1))$ be a $q$-twisted form in $\p{q}$ without zeros; this is possible since $\Omega^q_{\p{q}}(q+1) \cong \op{q}$. The set of critical points of $\pi|_{X}$ coincides with the zeros of $(\pi|_{X})^*\Theta$ which is a section of $H ^0(X, \Omega^q_{X}(q+1))$. As $q < n$, we have that this section cannot vanish in codimension one because, in this case, we would produce a non-zero section in $H^0(X, \Omega^q_{X}(r) )$, with $r\leq q$, but such a group is trivial.
\end{proof}

\begin{theorem}\label{T:degreoneSCI}
Let $X$ be a smooth complete intersection in $\p{n+s}$ of type $(d_1,\ldots,d_s)$, with $n\geq 4$ and $c_1(TX)\geq 2$. If either $s=1$ ($X$ is a hypersurface) or $d_j \geq 3$ for $j=1,\dots,s$, then the $\Folx{X}{1}{1}$ has two irreducible components:
\[
\Folx{X}{1}{1} = \Log(X;1,2) \cup \Log(X;1,1,1).
\]
\end{theorem}

\begin{proof}
Let $\F$ be a codimension-one foliation of degree one on $X$. By Corollary \ref{C:deg one}, we have that $\F$ is either defined by a closed rational 1-form without codimension-one zeros or $\F$ admits a subfoliation of codimension 2, and degree zero. In the first case,   $\F$ belongs to $\Log(X;1,2)$ or $\Log(X;1,1,1)$, see \cite[Lemma 2.5]{LPT13}.

Suppose there exists a subfoliation $\G \subset \F$ of codimension two and degree zero. By Theorem \ref{T:Ext2codimgeneral}, and Corollary \ref{C:Ext2codim} in the case of hypersurfaces, we have that $\G$ extends to $\p{n+s}$ and thus $\G$ is defined by the restriction to $X$ of a linear projection $\pi \colon \p{n+s} \dashrightarrow \p{2}$. By \cite[Lemma 3.1]{CLLPT7}, there exists a foliation $\Ho$ on $\p{2}$ such that $\F = (\restr{\pi}{X})^* \Ho$. Furthermore, Lemma \ref{L:crit-proj} guarantees that $\nF = (\restr{\pi}{X})^*\nH$. Therefore, $\Ho$ is a codimension-one foliation of degree one on $\p{2}$. The classification of foliations of degree one on $\p{2}$ {\cite[pp.8--19]{JP-Jouanoulou79}} implies that  $\F$ belongs to either $\Log(X;1,2)$ or $\Log(X;1,1,1)$.   
\end{proof}

%{
%Therefore, Theorem \ref{T: Main components}, {with $c_1(TX) \geq 2$}, follows from Theorem \ref{T:degreoneSCI} and Corollary \ref{C:ext-comp-int-deg>4}. Note that for $X\subset \p{n+1}$ a hypersurface, $c_1(TX) = n+2 - \deg(X)$. Thus, Theorem \ref{T:degreoneSCI} applies when $\deg(X) \leq n$. {The following Corollary  follows from Corollary \ref{C:ext-comp-int-deg>4}, which is \cite[Corolário 2.2.7]{TeseMateus}, and Theorem \ref{T:degreoneSCI}.}
%}
{
Therefore, Theorem~\ref{T: Main components}, under the assumption that $c_1(TX)\geq 2$, follows from Theorem~\ref{T:degreoneSCI} and Corollary~\ref{C:ext-comp-int-deg>4}. Note that if $X\subset \mathbb{P}^{n+1}$ is a hypersurface, then
\[
c_1(TX)=n+2-\deg(X).
\]
Hence, Theorem~\ref{T:degreoneSCI} applies whenever $\deg(X)\leq n$. Combining Corollary~\ref{C:ext-comp-int-deg>4} (see \cite[Corolário~2.2.7]{TeseMateus}) with Theorem~\ref{T:degreoneSCI} yields the following corollary.}

{
\begin{corollary}\label{C:degreonecomp}
Let $X$ be a smooth hypersurface of $\p{n+1}$, such that $X$ is neither a cubic threefold nor a three-dimensional quadric. Then irreducible components of the space of codimension-one foliations of degree one on $X$ are $\Log(X; 1,2)$ and $\Log(X;1,1,1)$.   
\end{corollary}
}

\begin{corollary}\label{C:resultado-Mateus}
Every codimension-one foliation of degree one on a smooth hypersurface of dimension at least four extends to $\p{n+1}$.
\end{corollary}

\begin{proof}
This follows by Corollary \ref{C:degreonecomp} and Proposition \ref{P:logext}.
\end{proof}

%{
%\begin{remark}
%It is worth noting that Corollary \ref{C:degreonecomp} and Corollary \ref{C:resultado-Mateus}, for the case of quadrics of dimension greater than or equal to 4 or hypersufaces with degree grether than 3, was proved in \cite[Teorema 2.2.5, Corolário 4.4.5, and Corolário 4.4.6]{TeseMateus}. 
%\end{remark}
%}

{
\begin{remark}
It is worth noting that Corollaries~\ref{C:degreonecomp} and~\ref{C:resultado-Mateus}, in the case of quadrics of dimension at least~4 or hypersurfaces of degree greater than~3, were proved in \cite[Theorem~2.2.5 and Corollaries~4.4.5 and~4.4.6]{TeseMateus}.
\end{remark}}

To complete the classification of $\Folx{X}{1}{1}$ for any smooth projective hypersurface $X$ of dimension $n\geq 3$, we only need to analyze the case of cubic hypersurfaces of dimension three. The case of a quadric threefold was described in \cite[Theorem 5.2]{LPT13}. In addition to the logarithmic components, there is an irreducible component whose general element is given by an action of the affine group $\mathfrak{aff}(\C)$. Furthermore, Corollary \ref{C:degreonecomp} provides us with the classification for the other hypersurfaces.

\subsubsection{Foliations on a cubic threefold}
Let $X\subset \p4$ be a smooth cubic hypersurface and $\F$ a codimension-one degree-one foliation on $X$ given by $\omega_0 \in H^0(X,\Omega_X^1(3))$. {The natural restriction map
\[
H^0(\p4,\Omega_\p4^1(3)) \lra H^0(X,\Omega_X^1(3))
\]
is an isomorphism, see \cite[Lemma 3.1]{Fig23}. Then $\omega_0$} extends to (a unique!) 1-form $\omega \in H^0(\p4,\Omega_\p4^1(3))$ defining a degree-one distribution on $\p4$. If $\omega$ is integrable then $\F$ is the restriction of a degree-one foliation on $\p4$, thus $\F$ belongs to either $\Log(X;2,1)$ or $\Log(X;1,1,1)$. We will prove this is always the case. 

\begin{theorem}\label{T:ext-deg1-cub3fold}
A codimension-one foliation of degree one on a smooth cubic threefold $X \subset \p4$ extends to a degree-one codimension-one foliation on $\p4$.
\end{theorem}

\begin{proof}
Let $\F$ be a codimension-one foliation of degree one on $X$,   and $\omega \in H^0(\p4,\Omega_\p4^1(3))$ be such that {$i^*\omega$ defines $\F$, for $i\colon X \hookrightarrow \p4$ the inclusion}. Suppose by contradiction that $\omega$ is not integrable, that is, $\omega\wedge d\omega \neq 0$. Then there exists a polynomial vector field $v \neq 0$ such that $\omega \wedge d\omega = \iota_{\rad}\iota_v (dx_0\wedge\dots\wedge dx_4)$, where $\rad = \sum_{j=0}^4x_j \del{ x_j}$ is the radial vector field. Let $f \in \CC[x_0, \dots, x_4]_3$ be a polynomial defining $X$. {Since $i^*(\omega\wedge d\omega)  = 0$, we have that $v(f) = fh$ for some polynomial $h$. Then, up to changing $v\mapsto v - \frac{1}{3}h\,\rad$, we may assume that $v(f) = 0$.} Note that, since $X$ is smooth, the partial derivatives of $f$ form a regular sequence. Hence, the Koszul complex associated with the gradient $\nabla (f)$ is exact, and $v$ must be of the form
\begin{equation} \label{E:campo-cubica}
    v = \sum_{i,j} \frac{\partial f}{\partial x_i} a_{ij} \del{ x_j},
\end{equation}
where $a_{ji} = -a_{ij}\in \CC$.  The action of  a change of coordinates by $\phi \in \GL_5$ transforms $v$ as follows:
\begin{align*}
    \phi_* v &= \sum_{i,j,l} \phi_{jl} \frac{\partial f}{\partial x_i}\circ \phi^{-1} a_{il} \del{ x_j} =   \sum_{i,j,l,m} \phi_{jl} \frac{\partial (f\circ \phi^{-1})}{\partial x_m}\phi_{mi}  a_{il} \del{ x_j} \\
    &=  \sum_{i,j}  \frac{\partial (f\circ \phi^{-1})}{\partial x_i}b_{ij} \del{ x_j},
\end{align*}
where $(b_{ij}) = \phi (a_{ij}) \phi^T$. Thus, up to a linear change of coordinates, we may assume that the matrix $(a_{ij})$ is in congruence normal form, yielding two possibilities for $v$:
\[
v_1 = \frac{\partial f}{\partial x_{1}} \frac{\partial }{\partial x_{0}}- \frac{\partial f}{\partial x_{0}} \frac{\partial }{\partial x_{1}} + \frac{\partial f}{\partial x_{3}} \frac{\partial }{\partial x_{2}}- \frac{\partial f}{\partial x_{2}} \frac{\partial }{\partial x_{3}} \quad \text{or} \quad  v_2 = \frac{\partial f}{\partial x_{1}} \frac{\partial }{\partial x_{0}}- \frac{\partial f}{\partial x_{0}} \frac{\partial }{\partial x_{1}}.
\]

We will show that none of those cases are possible. By construction,
\[
0 =  \iota_v (\omega\wedge d\omega) =  (\iota_v\omega)d\omega - \omega \wedge \iota_v d\omega,
\]
whence $\iota_v\omega = 0$, otherwise we would get $\omega \wedge d\omega = 0$. In particular, we get $\omega \wedge \iota_v d\omega =0$, hence $\iota_v d\omega = h\omega$ for some degree-one polynomial. In both cases of interest, the vector fields have vanishing divergence. Hence, $d\omega \wedge d\omega  = 3\iota_v(dx_0\wedge\dots\wedge dx_4)$. Moreover,
\[
0 = \iota_v(d\omega\wedge d\omega) = 2 (\iota_v d\omega)\wedge d\omega = 2h\, \omega \wedge d\omega.
\]
Therefore, $h = 0$ and $\iota_v d\omega = 0$. We will use this fact to derive a contradiction. 

\medskip
\noindent {\bf Case 1:} Note that the coefficients of $v_1$ form a regular sequence of length $4$, hence the associated Koszul complex is exact for polynomial $k$-forms with $2\leq k \leq 4$. Thus $\iota_{v_1} d\omega = 0$ implies that there exists a homogeneous $3$-form $\theta$ such that $d\omega = \iota_v \theta$. However, this is absurd since such $\theta$ would have coefficients of negative degree. Hence, $d\omega = 0$ which is also absurd.

% \medskip
% \noindent {\bf Case 2:} For $v_2$ the equation $\iota_{v_2}\omega =0$ implies that 
% \[
% \omega = \lambda \frac{\partial f}{\partial x_{0}} dx_0 + \lambda \frac{\partial f}{\partial x_{1}} dx_1 + a_2\,dx_2 +a_3\,dx_3 +a_4\,dx_4
% \]
% for some $\lambda \in \CC$ and $a_j$ of degree two. Using that $\frac{\partial f}{\partial x_{0}}, \frac{\partial f}{\partial x_{1}}$ is a regular sequence, we get from the equation $\iota_vd\omega = 0$ that, for $i = 0,1$ and $j=2,3,4$,
% \[
% \frac{\partial a_j}{\partial x_{i}} = \lambda \frac{\partial^2 f}{\partial x_{i} \partial x_{j}}.
% \]
% Therefore, $\omega = \lambda df + \eta$, where $\eta$ does not depend on $x_0$ or $x_1$.
% Note that $0 = \iota_{\rad}\omega = 3\lambda \,f + \iota_{\rad}\eta$, which implies that $\lambda =0$ and $\omega = \eta$. But this is absurd since $\eta$ is integrable. 
{
\medskip
\noindent {\bf Case 2:} For $v_2$ the equation $\iota_{v_2}\omega =0$ becomes $a_0\frac{\partial f}{\partial x_{1}} - a_1\frac{\partial f}{\partial x_{0}} = 0$. Since 
$\frac{\partial f}{\partial x_{0}}, \frac{\partial f}{\partial x_{1}}$ is a regular sequence and $\deg a_j = 2$, there exists $\lambda \in \CC$ such that 
\[
\omega = \lambda \frac{\partial f}{\partial x_{0}} dx_0 + \lambda \frac{\partial f}{\partial x_{1}} dx_1 + a_2\,dx_2 +a_3\,dx_3 +a_4\,dx_4.
\]
Taking the exterior derivative yields 
\begin{align*}
    d\omega &= \sum_{j\geq 2} \left(-\lambda \frac{\partial^2 f}{\partial x_{j} \partial x_{0}} + \frac{\partial a_j}{\partial x_{0}} \right)dx_0 \wedge dx_j + \sum_{j\geq 2} \left(-\lambda \frac{\partial^2 f}{\partial x_{j} \partial x_{1}} + \frac{\partial a_j}{\partial x_{1}} \right)dx_1 \wedge dx_j + \phi 
\end{align*}
where $\eta$ only depends on $dx_j$ for $j\geq 2$. Since $\frac{\partial f}{\partial x_{0}}, \frac{\partial f}{\partial x_{1}}$ is a regular sequence, it does not have linear syzygies. Then  $\iota_vd\omega = 0$ implies $d\omega -\phi =0$. In particular, for $i = 0,1$ and $j=2,3,4$,
\[
\frac{\partial a_j}{\partial x_{i}} = \lambda \frac{\partial^2 f}{\partial x_{i} \partial x_{j}}.
\]
Let $\eta = \omega - \lambda df$. Then the Lie derivative of $\eta$ in the direction of $\frac{\partial }{\partial x_{0}}$ vanishes for $i=0,1$. That means $\eta$ does not depend on $x_0$ or $x_1$. Note that $0 = \iota_{\rad}\omega = 3\lambda \,f + \iota_{\rad}\eta$, hence $\iota_{\rad}\eta = -3\lambda\,f$. Since $f$ depends on $x_0$ but $\eta$ does not, we have that $\lambda =0$ and $\omega = \eta$. In this case, $\eta$ depends on $3$ variables and $\iota_{\rad}\eta =0$. In particular, we have $\eta \wedge d\eta = 0$. However, this contradicts the choice of $\omega$.  
}
\end{proof}

\bibliographystyle{abbrvurl}
\bibliography{references}{}

\end{document}